\documentclass[12pt,british]{article}
\usepackage{amsmath,amsfonts,amssymb,amscd,amsthm}
\usepackage{babel}
\usepackage{enumerate,url,tikz-cd}
\usepackage{verbatim}
\usepackage[matrix,arrow]{xy}

\usepackage[pdftex,%
bookmarks=true,
breaklinks=true,
bookmarksnumbered = true,
colorlinks= true,
urlcolor= green,
anchorcolor = yellow,
citecolor=blue,
]{hyperref}                   

\usepackage{tikz}
\usepackage{tikz-cd}
\usetikzlibrary{arrows,decorations.markings,calc}
\tikzset{commutative diagrams/.cd,arrow style=tikz,diagrams={>=stealth'}}
\tikzset{middlearrow/.style={decoration={markings,mark=at position #1 with {\arrow[very thick]{stealth' reversed}}},postaction={decorate}}}
\tikzset{middlearrowrev/.style={decoration={markings,mark=at position #1 with {\arrow[very thick]{stealth'}}},postaction={decorate}}}

\renewcommand{\epsilon}{\ensuremath{\varepsilon}}
\renewcommand{\to}{\ensuremath{\longrightarrow}}
\newcommand{\R}{\ensuremath{\mathbb R}}
\newcommand{\T}{\ensuremath{\mathbb{T}^{2}}}
\newcommand{\rp}{\ensuremath{{\mathbb R}P^2}}

\newcommand{\N}{\ensuremath{\mathbb N}}

\newcommand{\Z}{\ensuremath{\mathbb Z}}

\newcommand{\dt}{\ensuremath{\mathbb D}^{\,2}}
\newcommand{\St}[1][2]{\ensuremath{\mathbb S}^{#1}}
\newcommand{\FF}{\ensuremath{\mathbb F}}
\newcommand{\F}[1][n]{\ensuremath{\FF_{{#1}}}}

\newcommand{\vide}{\ensuremath{\varnothing}}
\newcommand{\im}[1]{\ensuremath{\operatorname{\text{Im}}\left({#1}\right)}}
\renewcommand{\ker}[1]{\ensuremath{\operatorname{\text{Ker}}\left({#1}\right)}}

\makeatletter
\def\@enum@{\list{\csname label\@enumctr\endcsname}%
           {\usecounter{\@enumctr}\def\makelabel##1{
\normalfont\ignorespaces\emph{{##1}~}}
\setlength{\labelsep}{3pt}
\setlength{\parsep}{0pt}
\setlength{\itemsep}{0pt}
\setlength{\leftmargin}{0pt}
\setlength{\labelwidth}{0pt}
\setlength{\listparindent}{\parindent}%
\setlength{\itemsep}{0pt}
\setlength{\itemindent}{0pt}
\setlength{\topsep}{3pt plus 1pt minus 1 pt}}}
\renewcommand\theenumi{\@alph\c@enumi}
\renewcommand\theenumii{\@alph\c@enumii}
\renewcommand\theenumiii{\@alph\c@enumiii}
\renewcommand\theenumiv{\@alph\c@enumiv}

\renewcommand{\p@enumii}{}
\makeatother

\newcommand{\id}{\ensuremath{\operatorname{\text{Id}}}}
\newcommand{\lhra}{\mathrel{\lhook\joinrel\to}}

\newcommand{\splitmap}[3]{\operatorname{\text{Spl}}(#1,#2,#3)}

\DeclareRobustCommand*{\up}[1]{\textsuperscript{#1}}
\renewcommand{\th}{\ensuremath{\up{th}}}
\newcommand{\ft}[1][n]{\ensuremath{\Delta_{#1}^{2}}}

\newcommand{\brak}[1]{\ensuremath{\left\{ #1 \right\}}}
\newcommand{\ang}[1]{\ensuremath{\left\langle #1\right\rangle}}
\newcommand{\set}[2]{\ensuremath{\brak{#1 \,\mid\, #2}}}

\newcommand{\setr}[2]{\ensuremath{\brak{#1 \,\left\lvert \, #2 \right.}}}
\newcommand{\setl}[2]{\ensuremath{\brak{\left. #1 \,\right\rvert \, #2}}}

\newtheoremstyle{theoremm}{}{}{\itshape}{}{\scshape}{.}{ }{}

\theoremstyle{theoremm}
\newtheorem{thm}{Theorem}
\newtheorem{lem}[thm]{Lemma}
\newtheorem{prop}[thm]{Proposition}
\newtheorem{cor}[thm]{Corollary}

\newtheoremstyle{remarkk}{}{}{}{}{\scshape}{.}{ }{}

\theoremstyle{remarkk}

\newtheorem{rem}[thm]{Remark}
\newtheorem{rems}[thm]{Remarks}

\newtheoremstyle{citing}
  {}
  {}
  {\itshape}
  {}
  {\scshape}
  {.}
  {.5em}
  {\thmnote{#3}}

\theoremstyle{citing}

\newcommand{\reth}[1]{Theorem~\protect\ref{th:#1}}
\newcommand{\relem}[1]{Lemma~\protect\ref{lem:#1}}
\newcommand{\repr}[1]{Proposition~\protect\ref{prop:#1}}
\newcommand{\reco}[1]{Corollary~\protect\ref{cor:#1}}
\newcommand{\resec}[1]{Section~\protect\ref{sec:#1}}
\newcommand{\req}[1]{equation~(\protect\ref{eq:#1})}
\newcommand{\reqref}[1]{(\protect\ref{eq:#1})}
\newcommand{\rerem}[1]{Remark~\protect\ref{rem:#1}}

\newcommand{\ai}{\vbox to 7pt{\hbox to 7pt{\vrule height 7pt width 7pt}}}

\newcommand{\comj}[1]{\noindent\textbf{!!!J!!!~#1}}

\begin{document}

\title{Fixed points of $n$-valued maps on surfaces and the Wecken property~--~a configuration space approach}

\author{DACIBERG~LIMA~GON\c{C}ALVES\\
Departamento de Matem\'atica - IME - Universidade de S\~ao Paulo,\\
Rua do Mat\~ao, 1010 \\
 CEP 05508-090 - 
S\~ao Paulo - SP - Brazil.\\
e-mail:~\url{dlgoncal@ime.usp.br}\vspace*{4mm}\\
JOHN~GUASCHI\\
Normandie Universit\'e, UNICAEN,\\
Laboratoire de Math\'ematiques Nicolas Oresme UMR CNRS~\textup{6139},\\
CS 14032, 14032 Caen Cedex 5, France.\\
e-mail:~\url{john.guaschi@unicaen.fr}}


\date{13th February 2017, revised 16th March and 11th April.}

\begingroup
\renewcommand{\thefootnote}{}
\endgroup 

\maketitle

\begin{abstract}
\noindent 
\emph{In this paper, we explore the fixed point theory of $n$-valued maps using configuration spaces and braid groups, focussing on two fundamental problems, the Wecken property, and the computation of the Nielsen number. We show that the projective plane (resp.\ the $2$-sphere $\St$) has the Wecken property for $n$-valued maps for all $n\in \N$ (resp.\ all $n\geq 3$). In the case $n=2$ and $\St$, we prove a partial result about the Wecken property. We then describe the Nielsen number of a non-split $n$-valued map $\phi\colon\thinspace X \multimap X$ of an orientable, compact manifold without boundary in terms of the Nielsen coincidence numbers of a certain finite covering $q\colon\thinspace \widehat{X} \to X$ with a subset of the coordinate maps of a lift of the $n$-valued split map $\phi\circ q\colon\thinspace \widehat{X} \multimap X$.} 
\end{abstract}

%
%
%
%

\section{Introduction}\label{sec:intro}

Multifunctions and their fixed point theory have been widely studied, see the books~\cite{Be,Gor} for example, where fairly general classes of multifunctions and spaces are considered. In all of what follows, $X$ and $Y$ will be topological spaces. Let $\phi\colon\thinspace X \multimap Y$ be an $n$-valued function \emph{i.e.}\ a function that to each $x\in X$ associates an unordered subset $\phi(x)$ of $Y$ of cardinality $n$. Recall that such an $n$-valued function $\phi$ is \emph{continuous} if for all $x\in X$, $\phi(x)$ is closed, and for any open set $V$ in $Y$, the sets $\set{x\in X}{\phi(x)\subset V}$ and $\set{x \in  X}{\phi(x)\cap V \neq  \varnothing}$ are open in $X$. We will refer to a continuous $n$-valued function as an \emph{$n$-valued map}.  The class of $n$-valued maps is of particular interest, and more information about their fixed point theory on finite complexes may be found in~\cite{Bet1,Bet2,Brr1,Brr2,Brr3,Brr4,Brr5,Brr6,Sch0,Sch1,Sch2}. A \emph{homotopy} between two $n$-valued maps $\phi_1,\phi_2\colon\thinspace X \multimap Y$ is an $n$-valued map $H\colon\thinspace  X\times I \multimap Y$ such that $\phi_1=H ( \cdot , 0)$ and  $\phi_2=H ( \cdot , 1)$. Following~\cite{Sch0}, an $n$-valued function $\phi\colon\thinspace  X \multimap Y$ is said to be a \emph{split $n$-valued map} if there exist single-valued maps $f_1, f_2, \ldots, f_n\colon\thinspace X \to Y$ such that $\phi(x)=\brak{f_1(x),\ldots,f_n(x)}$ for all $x\in X$. This being the case, we shall write $\phi=\brak{f_1,\ldots,f_n}$. Let $\splitmap{X}{Y}{n}$ denote the set of split $n$-valued maps between $X$ and $Y$. \emph{A priori}, $\phi\colon\thinspace X \multimap Y$ is just an $n$-valued function, but if it is split then it is continuous by \cite[Proposition~42]{GG15}, which justifies the use of the word `map' in the definition. Partly for this reason, split $n$-valued maps play an important r\^ole in the theory. If $\phi\colon\thinspace X \multimap X$ is an $n$-valued map from $X$ to itself, we say that $x\in X$ is a \emph{fixed point} of $\phi$ if $x\in \phi(x)$, and we denote the set of fixed points of $\phi$ by $\operatorname{\text{Fix}}(\phi)$.

In~\cite[Section~5]{Sch1}, Schirmer defined the notion of Nielsen number of $n$-valued maps of finite complexes. Her definition is similar to that for single-valued maps, although it is a little more elaborate. As for the case of single-valued maps, for appropriate spaces, the Nielsen number $N(\phi)$ of an $n$-valued map $\phi\colon\thinspace X \multimap X$ provides a lower bound for the number of fixed points among all $n$-valued maps homotopic to $\phi$. The computation of the Nielsen number of a self- or $n$-valued map is an important problem in fixed point theory, and is not easy in general. In the split case, we have the following formula for the Nielsen number of $n$-valued maps of polyhedra in terms of the constituent single-valued maps.
\begin{thm}\label{th:helgath0}\cite[Corollary~7.2]{Sch1}
Let $\phi=\brak{f_1,f_2,\ldots, f_n}\colon\thinspace  X \multimap X$ be a split $n$-valued map, where $X$ is a compact polyhedron. Then $N(\phi)=N(f_1)+\cdots+N(f_n)$.
\end{thm} 

A second fundamental problem in fixed point theory is to decide whether a space $X$ has the Wecken property. Recall that the homotopy class of an $n$-valued map $\phi\colon\thinspace X \multimap X$ is said to have the \emph{Wecken property} if there exists an $n$-valued map $\psi\colon\thinspace X \multimap X$ homotopic to $\phi$ that has exactly $N(\phi)$ fixed points, and that a space $X$ has the \emph{Wecken property for $n$-valued maps} if every homotopy class of $n$-valued maps of $X$ has the Wecken property.
%
%
%
For single-valued maps, many complexes of dimension at least three have the Wecken property~\cite{Ke,Wec1,Wec2,Wec3}. In the case of surfaces, the $2$-sphere $\St$, the real projective plane $\rp$ have the Wecken property~\cite{BGZ,Ji}, as do the $2$-torus $\T$ and the Klein bottle (see~\cite{GK} for example).  However, Jiang showed that no other compact surface without boundary has the Wecken property~\cite{Ji0,Ji1}. For $n$-valued maps,   substantial progress has been made in the study of the Wecken property. The following result mirrors that for single-valued maps.
\begin{thm}\label{th:helgath1}\cite[Theorem~5.2]{Sch2}
Let $M$ be a compact triangulable manifold (with or without boundary) of dimension greater than or equal to $3$, and let $n\in \N$. Then every $n$-valued map $\phi \colon\thinspace M \multimap M$ is homotopic to an $n$-valued map that has $N(\phi)$ fixed points. In particular, $M$ has the Wecken property for $n$-valued maps.
\end{thm}

\reth{helgath1} has been extended to a larger class of spaces, including some $2$-dimensional complexes, as follows.
 
\begin{thm}\label{th:bete}\cite[Theorem~1]{Bet2}
Let $X$ be a compact polyhedron without local cut points and such that no connected component of $X$ is a surface.
Then $X$ has the Wecken property for $n$-valued maps.
\end{thm}

In a recent paper~\cite{GG15}, we studied some aspects of fixed point theory of $n$-valued maps from $X$ to $Y$ by introducing an equivalent and natural formulation in terms of single-valued maps from $X$ to the $n\up{th}$ unordered configuration space $D_n(Y)$ of $X$, where $D_n(Y)$ is the quotient of the $n\th$ (ordered) configuration space $F_{n}(X)$ of $X$, defined by:
\begin{equation*}
F_n(Y)=\setr{(y_1,\ldots,y_n)}{\text{$y_i\in Y$, and $y_i\neq y_j$ if $i\neq j$}},
\end{equation*}
by the free action of the symmetric group $S_{n}$ given by permuting coordinates. It is well known that $\pi_{1}(D_{n}(Y))$ (resp.\ $\pi_{1}(F_{n}(Y))$) is the braid group $B_{n}(Y)$ (resp.\ pure braid group $P_{n}(Y)$) of $Y$ on $n$ strings, that the quotient map $\pi\colon\thinspace F_{n}(Y) \to D_{n}(Y)$ is a regular $n!$-fold covering, and that $P_{n}(Y)$ is the kernel of the surjective homomorphism $\tau\colon\thinspace B_n(Y) \to S_n$ that to a braid associates its induced permutation. Configuration spaces play an important r\^ole in several branches of mathematics and have been extensively studied, see~\cite{CG,FH} for example. As in~\cite{GG15}, a map $\Phi\colon\thinspace X \to D_n(Y)$ will be called an \emph{$n$-unordered map}, and a map $\Psi\colon\thinspace X \to F_n(Y)$ will be called an \emph{$n$-ordered map}. For such an $n$-ordered map, for $i=1,\ldots,n$, there exist maps $f_i\colon\thinspace X \to Y$ such that $\Psi(x)=(f_1(x),\ldots, f_n(x))$ for all $x\in X$, and for which $f_i(x)\neq f_j(x)$ for all $1\leq i,j\leq n$, $i\neq j$, and all $x\in X$. In this case, we will often write $\Psi=(f_{1},\ldots,f_{n})$. There is an obvious bijection between the set of $n$-point subsets of $Y$ and $D_{n}(Y)$ that induces a bijection between the set of $n$-valued functions from $X$ to $Y$ and the set of functions from $X$ to $D_{n}(Y)$. If we suppose in addition that $X$ and $Y$ are metric spaces, then $D_{n}(Y)$ may be equipped with a certain Hausdorff metric~\cite[Appendix]{GG15}, in which case this bijection restricts to a bijection between the set of $n$-valued maps from $X$ to $Y$ and the set of (continuous) maps from $X$ to $D_{n}(Y)$~\cite[Theorem~8]{GG15}. If $\phi\colon\thinspace X \multimap Y$ is an $n$-valued map then we shall say that the map $\Phi\colon\thinspace X \to D_{n}(Y)$ obtained by this bijection as the map associated to $\phi$. If a map $\Phi\colon\thinspace  X \to D_n(Y)$ admits a lift $\widehat{\Phi}\colon\thinspace  X \to F_n(Y)$ via the covering map $\pi$ then we say that $\widehat{\Phi}$ is a \emph{lift} of $\phi$. By~\cite[Section~2.1]{GG15}, $\phi$ is split if and only if it admits a lift.

In~\cite{GG15}, we showed that spheres and real and complex projective spaces of even dimension have the fixed point property for $n$-valued maps, and we expressed the problem of deforming an $n$-valued map to a fixed point free $n$-valued map in terms of an algebraic criterion involving the braid groups of $X$, which we then used to determine an infinite number of homotopy classes of split $2$-valued maps of $\T$ that contain a fixed point free representative. In this paper, we explore some other aspects of fixed point theory of $n$-valued maps using the above formulation in terms of configuration spaces. We establish an equality for the Nielsen number in the non-split case, and we study the Wecken property for some surfaces.
We now describe the contents of this paper in more detail. 
 
In \resec{pres}, we recall some fundamental results from~\cite{GG15}, and in \resec{wecdisc}, we show that the $2$-disc has the Wecken property for $n$-valued maps for all $n\in \N$. In Sections~\ref{sec:nvalS2} and~\ref{sec:proj}, we analyse the Wecken problem for $n$-valued maps of $\St$ and $\rp$ respectively. In the first case, one important step is to determine the number of homotopy classes of $n$-valued maps of $\St$. In \relem{sphWec}, we show that there is just one such class if $n\geq 3$, while if $n=2$, we show that the set of such homotopy classes is in bijection with $\N$. If $\phi\colon\thinspace \St \multimap \St$ is a $2$-valued map then we refer to the non-negative integer given by this correspondence as the \emph{degree} of $\phi$ (or its homotopy class). Let $A\colon\thinspace \St \to \St$ denote the antipodal map. In the case of $\rp$, in \repr{classmap}, we show that there are precisely two homotopy classes of $n$-valued maps of $\rp$ for all $n\geq 2$, and we give explicit representatives of these homotopy classes. To prove the Wecken property for $n$-valued maps of $\rp$, it then suffices to check that each of these homotopy classes has the Wecken property. We summarise the main results of Sections~\ref{sec:pres} and~\ref{sec:proj} as follows. 

\begin{prop}\label{prop:propscst}\mbox{} 
\begin{enumerate}[(a)]
\item\label{it:propscsta} Let $f_0\colon\thinspace\St \to \St$ be the constant map at a point $x_{0}\in \St$. Then $f_{0}$ has degree $0$, and $f_{0}$ and $A\circ f_0$ each have precisely one  fixed point.

\item\label{it:propscstb} There exists a map $ f_1\colon\thinspace \St \to \St$ of degree $1$ that has a single fixed point and such that
the map $A\circ f_1$ is fixed point free.
\item\label{it:propscstc}  There exists a map $ f_2\colon\thinspace \St \to \St$ of degree $2$ such that $f_{2}$ and the map $A\circ f_2$ each possess a single fixed point.
\item \label{it:propscstd}
The Wecken property holds for:
\begin{enumerate}[(i)]
\item\label{it:propscstb1} $n$-valued maps of $\St$ for all $n\geq 3$. 
\item\label{it:propscstb2} the homotopy classes of $2$-valued maps of degree $0$, $1$ or $2$. 
\end{enumerate}
\end{enumerate}
\end{prop}

\repr{propscst}(\ref{it:propscstd})(\ref{it:propscstb2}) has been obtained independently in~\cite{Brr0}.
 Our proof is different from that given  in~\cite{Brr0}, and is perhaps more direct. The paper~\cite{Brr0} contains other results about $n$-valued
 maps of $\St$, concerning for example the minimal number of fixed points in a given homotopy class of a $2$-valued map (see~\cite[Theorem 4.1]{Brr0}). Apart from the three cases given in~\repr{propscst}(\ref{it:propscstd})(\ref{it:propscstb2}), we believe that this minimal number is not known. We do not know either whether the Wecken property holds for $2$-valued maps  of $\St$ of degree greater than or equal to $3$. 


\begin{thm} \label{th:wecproj}
The projective plane $\rp$ has the Wecken property for $n$-valued maps for all $n\in \N$.
\end{thm}

In \resec{calcu}, we consider the problem of the computation of the Nielsen number of $n$-valued maps of a compact, orientable manifold $X$ without boundary. To do this, we shall use~\cite[Proposition~16]{GG15}, which given an $n$-valued map $\phi\colon \thinspace X \multimap X$, states that there exists a finite covering $q\colon\thinspace \widehat{X} \to X$ such that the composition $\phi\circ q \colon \thinspace \widehat{X} \multimap X$ is split, and~\cite[Proposition~17]{GG15}, which describes the fixed points of $\phi$ in terms of the coincidences of $q$ with the coordinate maps of a lift of the split $n$-valued map $\phi\circ q$. We analyse the behaviour of this description with respect to the essential Nielsen classes (\emph{i.e.}\ the Nielsen classes of non-zero index).  This yields a partial analogue of \reth{helgath0} in the non-split case as follows. 


\begin{thm} \label{th:form} 
Let $X$ be an orientable, compact manifold without boundary, and let $\phi\colon\thinspace X \multimap X$ be a non-split $n$-valued map. Let $\Phi\colon\thinspace X \to D_{n}(X)$ be the map associated to $\phi$, let $H$ be the kernel of the composition $\tau\circ \Phi_{\#} \colon\thinspace \pi_{1}(X)\to S_{n}$, let $L'=\im{\tau\circ \Phi_{\#}}$, and for $i=1,\ldots,n$, let $L_{i,i}'=\setl{\alpha\in L'}{\alpha(i)=i}$. Let $q\colon\thinspace \widehat{X} \to X$ be the finite regular covering of $X$ that corresponds to $H$. Let $\widehat{\Phi}_{1}=(f_{1},\ldots,f_{n}) \colon\thinspace \widehat{X} \to F_{n}(X)$ be a lift of the split $n$-valued map $\phi\circ q \colon \thinspace \widehat{X} \multimap X$, and suppose that the action of $L_{i,i}'$ on $\{1,\dots,n\}$ is free for all $i\in \{1,\dots,n\}$.
 Then the Nielsen number $N(\phi)$ of $\phi$ is equal to $\sum_{j\in I_0} N(q,f_{i_j})$, where $i_j$ runs over a set $I_{0}$ of representatives of the orbits of this action.
\end{thm}

If $n=2$, the condition that the action be free is necessarily satisfied, and in this case, we shall prove in \reco{nielsenphi} that $N(\phi)=N(q, f_1)=N(q, f_2)$. The study of the Wecken property for multivalued maps of the torus constitutes work in progress by the authors, and some preliminary results that are closely related to this problem may be found in~\cite{GG15}.

\subsection*{Acknowledgements}

The first-named author was partially supported by FAPESP-Funda\c c\~ao de Amparo a Pesquisa do
Estado de S\~ao Paulo, Projeto Tem\'atico Topologia Alg\'ebrica, Geom\'etrica  e Diferencial 2012/24454-8. The second-named author was also partially supported by the same project as well as the CNRS/FAPESP PRC project n\up{o}~275209 during his visit to the Instituto de Matem\'atica e Estat\'istica, Universidade de S\~ao Paulo, from the 4\textsuperscript{th} to the 22\textsuperscript{nd} of February 2017.


\section{Preliminaries and the Wecken property for $n$-valued maps
 of the $2$-disc and $2$-sphere}\label{sec:pres}

Let $n\in \N$. Given a topological space $X$, we are interested in understanding the fixed point theory of its $n$-valued maps. In this section, we will show that if $n\in \N$ (resp.\ $n\geq 3$), the $2$-disc $\dt$ and the (resp.\ the $2$-sphere $\St$) has the Wecken property for $n$-valued maps. We first recall some definitions and results from~\cite{GG15}. Given an $n$-valued map $\phi\colon\thinspace X \multimap X$ as defined in \resec{intro}, let $\Phi \colon\thinspace X \to D_n(X)$ denote the associated $n$-unordered map, and if $\phi$ is split, let $\widehat{\Phi}\colon\thinspace X \to F_n(X)$ be an $n$-ordered map that is a lift of $\phi$.
By a single-valued map (or simply a map) $f\colon\thinspace X \to Y$ between two topological spaces, we shall always mean a continuous function from $X$ to $Y$, and $[X,Y]$ will denote the set of unbased homotopy classes of maps between $X$ and $Y$. We recall the following result, which shows that for a large class of spaces, (homotopy classes of) $n$-valued maps may be identified with (homotopy classes of) maps whose target is an unordered configuration space. Let $I$ denote 
the unit interval $[0,1]$.
\begin{thm}\label{th:metriccont}
Let $X$ and $Y$ be metric spaces, and let $n\in \N$.
\begin{enumerate}[(a)]
\item\label{it:metricconta} An $n$-valued function $\phi\colon\thinspace X \multimap  Y$ is continuous if and only if the corresponding function $\Phi\colon\thinspace  X \to D_n(Y)$ is continuous.
\item\label{it:metriccontb} The set of homotopy classes of $n$-valued maps from $X$ to $Y$
is in one-to-one correspondence with the set $[X,D_{n}(Y)]$ of homotopy classes of maps from $X$ to $D_n(Y)$.
\end{enumerate}
\end{thm}

\begin{proof}
Part~(\ref{it:metricconta}) is a direct consequence of~\cite[Theorem~8]{GG15}.
For part~(\ref{it:metriccontb}), since $X\times I$ is also a metric space, we may apply~\cite[Theorem~8]{GG15}
to $n$-valued maps between $X\times I$ and $Y$.
So it follows that two $n$-valued maps $\phi_1, \phi_2\colon\thinspace X\multimap Y$
are homotopic if and only if the corresponding maps $\Phi_1, \Phi_2\colon\thinspace X \to D_{n}(Y)$
are homotopic, and the result follows.
\end{proof}

\begin{rem}\label{corresp}
\reth{metriccont} also holds under weaker hypotheses.
\begin{enumerate}[(a)]
\item If we assume just that $Y$ is a metric space, then the statement of \reth{metriccont}(\ref{it:metricconta}) is~\cite[Corollary~4.1]{Brr7}, and that of \reth{metriccont}(\ref{it:metriccontb}) follows by applying this corollary to the spaces $X\times I$ and $Y$.

\item If $X$ is just locally path-connected and semi-locally simply connected then so is $X\times I$. The statement of \reth{metriccont}(\ref{it:metricconta}) is~\cite[Corollary~3.1]{Brr7}, and that of \reth{metriccont}(\ref{it:metriccontb}) follows by applying this corollary to the spaces $X\times I$ and $Y$. 
\end{enumerate}
\end{rem}

From now on, we will assume that our spaces are metric spaces, so that we may interpret an $n$-valued map from $X$ to $Y$ as a map from $X$ to $D_n(Y)$ using \reth{metriccont}(\ref{it:metricconta}).


\subsection{The case of  $n$-valued maps of the $2$-disc}\label{sec:wecdisc}

\begin{prop}\label{prop:fpp}
The $2$-disc $\dt$ has the Wecken property for $n$-valued maps for all $n\in \N$.
\end{prop}

\begin{proof}
Let $n\in \N$, and let $(x_{1},\ldots,x_{n})\in F_{n}(\dt)$. By Theorem \ref{th:metriccont}(\ref{it:metriccontb}),  the set of homotopy classes of $n$-valued maps of the $2$-disc $\dt$ may be identified with the set $[\dt, D_n(\dt)]$ of homotopy classes of maps between $\dt$ and $D_{n}(\dt)$. Since the domain $\dt$ is contractible, this set has only one element, which is the class of the constant map $c=\brak{c_{1},\ldots,c_{n}}
\colon\thinspace \dt \to D_n(\dt)$, where for all $i=1,\ldots,n$, the map $c_i\colon\thinspace \dt \to \dt$ is the constant map at $x_i$. Then the map $c$ has exactly $n$ fixed points, and $N(c)=n$ by \reth{helgath0}, which proves the proposition.
\end{proof}

\subsection{The case  of $n$-valued maps of the $2$-sphere}\label{sec:nvalS2}

We now show that the Wecken property holds for $n$-valued maps of $\St$ for all $n\neq 2$. The case $n=2$ remains open, although we are able to provide some partial results (see also~\cite{Brr0}). Since the case $n=1$ is well known, see for example~\cite[Proposition~2.2]{BGZ},
we will assume that $n\geq 2$.

\begin{lem}\label{lem:sphWec}
If $n>2$ (resp.\ $n=2$), the set of homotopy classes of $n$-valued maps of $\St$ possesses exactly one element (resp.\ is in one-to-one correspondence with $\N$).
\end{lem}

\begin{proof}
\reth{metriccont}(\ref{it:metriccontb}) 
implies that it suffices to determine the set $[\St, D_n(\St)]$, which in turn is the orbit space of $\pi_{2}(D_{n}(\St))$ under the action of $\pi_1(D_{n}(\St))$. Now $F_{n}(\St)$ is a regular $n!$-fold covering of $D_{n}(\St)$, and if $n=2$ (resp.\ $n\geq 3$), the universal covering of $F_{n}(\St)$ has the homotopy type of $\St$ (resp.\ of $\St[3]$) by~\cite{BCP,FZ} or~\cite[pp.~43--44]{GG10}. Using standard results about homotopy groups and covering spaces, it follows that $\pi_{2}(D_{n}(\St))$ is isomorphic to $\Z$ (resp.\ is trivial) if $n=2$ (resp.\ $n>2$) (see also~\cite[Corollary, page~211]{FvB} for the case $n\geq 3$). This proves the result if $n\geq 3$. If $n=2$, observe that the map $\St \to F_2(\St)$ given by $x \longmapsto (x,-x)$ is a homotopy equivalence that is $\Z_2$-equivariant with respect to the action of the antipodal map on $\St$ and the action on $F_2(\St)$ given by $s\colon\thinspace F_2(\St) \to F_2(\St)$ defined by $s(x,y)=(y,x)$. This gives rise to a homotopy equivalence between the corresponding orbit spaces, namely $\rp$ and $D_2(\St)$. Since the action of $\pi_1(\rp)\cong \Z_2$ on $\pi_2(\rp) \cong \Z$ is multiplication by $-1$, the same is true for the action of $\pi_1(D_2(\St)$ on $\pi_2(D_2(\St))$, and the orbits are the subsets of the form $\brak{m,-m}$, where $m\in \Z$. It follows that the set $[\St, D_2(\St)]$ is in bijection with $\N$.
\end{proof}


%

We now study the Wecken property for $n$-valued maps of $\St$. 
Recall that if $\phi\colon\thinspace \St \multimap \St$ is a $2$-valued map, the integer given by the correspondence with $\N$ of \relem{sphWec} is the \emph{degree} of $\phi$ (or of its homotopy class). 

\begin{proof}[Proof of \repr{propscst}]\mbox{}
\begin{enumerate}[(a)]
\item Part~(\ref{it:propscsta}) is clear since the constant map $f_{0}\colon\thinspace\St\to \St$ at a point $x_{0}\in \St$ has precisely one fixed point, namely $x_{0}$, and $-x_{0}$ is the unique fixed point of $A\circ f_{0}$. 
\item Let $f_1$ be a self-map of the unit $2$-sphere $\St$ that is a small deformation of the identity, \emph{i.e.}\ $f_1$ satisfies $\left\lvert x-f_1(x)\right\rvert<\pi/2$ for all $x\in \St$, and that has exactly one fixed point $x_{0}$. Such a map may be constructed using a vector field on $\St$ that possesses just one singular point, and in this way, the degree of $f_{1}$ is equal to $1$. 
Then $f_1(x)\neq -x$ for all $x\in \St$, and thus $A\circ f_1$ is fixed point free, which concludes the proof of  part~(\ref{it:propscstb}).  

\item Consider the self map $\rho\colon\thinspace \St[1]\to \St[1]$ given by $z\longmapsto z^2$. This map has one fixed point, which is $\{1\}$, and the map $A\circ \rho$ has a single fixed point,
which is $\{-1\}$. 
Now consider the reduced suspension of $\St[1]$, \emph{i.e.}\
$(\St[1]\times [0,1])/(\{1\}\times [0,1]\cup \St[1]\times \{0,1\})=\St$, 
and let $f_2$ be the suspension 
of the map $\rho$. Then $f_2$  has only one fixed point which is the equivalence class of the point $\{1\}\times \{0\}$, 
and the map $A\circ f_2$ has a single fixed point,
which is the equivalence class of $\{-1\}\times \{1/2\}$, so it has the desired property. 
\item To prove part~(\ref{it:propscstb1}), by \relem{sphWec}, there is only one homotopy class of $n$-valued maps of $\St$ if $n\geq 3$, which is that of the constant $n$-valued map. The result then follows from part~(\ref{it:propscsta}).
To prove part~(\ref{it:propscstb2}), the case of the homotopy class of degree $0$ follows as in part~(\ref{it:propscstb1}). Now let $i\in \brak{1,2}$. For the homotopy class of degree $i$,  consider the split $2$-valued map $\phi_i=\brak{f_i,A\circ f_i}\colon\thinspace \St \multimap \St$. Then $L(f_i)=1+i$, so $N(f_i)\neq 0$, and thus $N(f_i)=1$ because $\St$ is simply connected. Hence by \reth{helgath0},
$N(\phi_i)=N(f_i)+N(A\circ f_i)$, and so $N(\phi_1)=1+0=1$ and $N(\phi_2)=1+1=2$. Since $\operatorname{\text{Fix}}(\phi_1)=\brak{x_{0}}$, where $x_{0}$ is the fixed point of $f_{1}$ given
in the proof of part~(\ref{it:propscstb}),
and $\operatorname{\text{Fix}}(\phi_2)=\brak{1, -1}$, it follows 
that the 
map $\phi_i$ has the Wecken property.\qedhere
\end{enumerate}
\end{proof}

\section{The case of $n$-valued maps of the projective plane}\label{sec:proj}

The aim of this section is to prove that the projective plane $\rp$ has the Wecken property for $n$-valued maps for all $n\in \N$.
Jiang showed that $\rp$ has the Wecken property for single-valued maps~\cite{Ji}, so it will suffice to study the case $n\geq 2$. We start by computing the Nielsen number of an $n$-valued map of $\rp$. 

\begin{lem}\label{lem:rp2split}
Let $n\geq 1$, and let $\phi\colon\thinspace \rp\multimap \rp$ be an $n$-valued map. Then $N(\phi)=n$.
\end{lem}

\begin{proof}
From~\cite[Lemma~14]{GG15}, $\phi$ is split, so for $i=1,\ldots,n$, there exist self-maps $f_i\colon\thinspace\rp \to \rp$  such that $\phi=\brak{f_1,\ldots,f_n}$. Applying \reth{helgath0}, $N(\phi)=n$ since $N(f_i)=1$ for all $i=1,\ldots,n$ by~\cite{Ji}.
\end{proof}


We now classify the homotopy classes of $n$-valued maps of $\rp$. If $n=1$, it is well known that the set $[\rp, \rp]$ of homotopy classes of self-maps of $\rp$ that induce the trivial homomorphism on the level of fundamental groups has two elements, see~\cite[Proposition~2.1]{GS} for example. One of these two homotopy classes is that of the constant map. We will describe the second homotopy class in terms of a  representative map $W_{P}\colon\thinspace \rp\to \rp$ that we shall now define, where $P\in \St$, and $\St\subset \R^3$ is the unit sphere in $\R^3$, which we equip with spherical coordinates $(\theta,\varphi)$, where $\theta\in [0,2\pi)$ and $\varphi\in [-\pi/2,\pi/2]$, so that $P=(\theta,\pi/2)$. 
With respect to the Cartesian coordinate system for which $P=(0,0,1)$ and the point $(1,0,0)$ has spherical coordinates $(0,0)$, the point with spherical coordinates $(\theta,\varphi)$ has Cartesian coordinates $(\cos \varphi \cos\theta, \cos \varphi \sin\theta,\sin \varphi)$. From this, one may see that with respect to the spherical coordinate system, $(\theta,\varphi-\frac{\pi}{2})$ may be identified with $(\theta+\pi, -\varphi-\frac{\pi}{2})$. We regard $\rp$ as the quotient of $\St$ by the free action of the group generated by the antipodal map $A$. Let $p\colon\thinspace \St \to \rp$ be the usual covering map, for all $x\in \St$, let $\overline{x}=p(x)$, and let $H_P^{+}$ be the hemisphere of $\St$ whose pole is $P$. 
 Let $U_{P}\colon\thinspace \St\to \St$ be the map defined by $U_{P}(\theta,\varphi)=(\theta,2\varphi-\frac{\pi}{2})$. The restriction $U_{P}\bigl\lvert_{H_{P}^{+}}\bigr. \colon\thinspace H_P^{+} \to \St$ sends each semi-meridian lying in $H_P^{+}$ that starts at $P$ and ends at the equator linearly to the meridian of $\St$ that starts at $P$, ends at $-P$ and contains the original semi-meridian. Since $U_{P}$ sends the whole of the equator to the point $-P$, $U_{P}\bigl\lvert_{H_{P}^{+}}\bigr.$ induces a map $W_P'\colon\thinspace \rp \to \St$ defined by $W_{P}'(\overline{x})= U_{P}(x)$ for all $x\in H_{P}^{+}$. Let $W_{P}\colon\thinspace \rp\to \rp$ be defined by $W_{P}=p\circ W_{P}'$. Now $A(\theta,\varphi)=(\theta+\pi,-\varphi)$, and up to the above-mentioned identification within the spherical coordinate system, for all $(\theta,\varphi)\in \St$, we have:
\begin{align*}
 U_{P}\circ A(\theta,\varphi)&= \textstyle U_{P}(\theta+\pi,-\varphi)=(\theta+\pi,-2\varphi-\frac{\pi}{2})= (\theta,2\varphi-\frac{\pi}{2}) =  U_{P}(\theta,\varphi),   
\end{align*}
and it follows that $U_{P}$ is a lift of $W_{P}$. We thus have the following commutative diagram: 
\begin{equation*}
\begin{tikzcd}
H_{P}^{+} \ar[hookrightarrow]{r} \ar[swap]{rd}{U_{P}\bigl\lvert_{H_{P}^{+}}\bigr.} & \St \ar{d}{U_{P}} \ar{r}{p} & \rp \ar{d}{W_{P}}\\
& \St \ar{r}{p} & \rp,
\end{tikzcd}
\end{equation*}
for which $W_{P}=p\circ W_{P}'$ also. The following lemma summarises various properties of $W_{P}$.

\begin{lem}\label{lem:prinWe}
Let $P\in \St$. The map $W_{P}\colon\thinspace \rp \to \rp$ satisfies the following properties:
\begin{enumerate}[(a)]
\item\label{it:prinWea} $W_{P}=W_{-P}$, and $\operatorname{\text{Fix}}(W_{P})=\brak{p(P)}$.
\item\label{it:prinWeb} The map $W_{P}$ is non null-homotopic, so it belongs to the non-constant homotopy class of $[\rp, \rp]$ that induces  the trivial homomorphism on the fundamental group.
\item\label{it:prinWed} Let $P_{1},P_{2}\in \St$. If $P_1\neq \pm P_2$, then $\operatorname{\text{Coin}}(W_{P_1}, W_{P_2})$ 
is empty.
\item\label{it:prinCoin} If $c_0\colon\thinspace \rp \to \rp$ is a constant map then the pair $(W_P, c_0)$ cannot be deformed to a coincidence-free pair.
\end{enumerate}
\end{lem}

  
\begin{proof}\mbox{}
\begin{enumerate}[(a)]
\item Let $\overline{x}\in \rp$, where $\overline{x}=p(x)$ and we take $x$ to belong to $H_{P}^{+}$. Thus $-x\in H_{-P}^{+}$, $p(-x)=\overline{x}$, and so $W_{-P}'(\overline{x})=U_{-P}(-x)$. With respect to the same spherical coordinate system that was used to define $U_{P}$, one may see 
that $U_{-P}(\theta,\varphi)= (\theta,2\varphi+\frac{\pi}{2})$ for all $(\theta,\varphi) \in \St$. From the above definitions, if $x=(\theta,\varphi)$, we have:
\begin{equation*}
W_{P}(\overline{x}) \textstyle=p\circ W_{P}'(\overline{x})= p\circ U_{P}(x)= \textstyle p(\theta,2\varphi-\frac{\pi}{2}), 
\end{equation*}
from which it follows that:
\begin{align*}
W_{-P}(\overline{x}) & \textstyle=p\circ W_{-P}'(\overline{x})= p\circ U_{-P}(-x)= p\circ U_{-P}(\theta+\pi,-\varphi)= p(\theta+\pi, -2\varphi+\frac{\pi}{2})\\
& \textstyle=p\circ A(\theta,2\varphi-\frac{\pi}{2})=p(\theta,2\varphi-\frac{\pi}{2})=W_{P}(\overline{x}),
\end{align*}
hence $W_{P}=W_{-P}$ as required.
\item Since $W_P$ factors through $\St$, it induces the trivial homomorphism on the fundamental group of $\rp$. Further, the map $W_P'$ is a lift of $W_P$, so is non null-homotopic (it represents the non-trivial element of $[\rp, \St]$) because its absolute degree is congruent to $1 \bmod{2}$.
\item Let $P_{1},P_{2}\in \St$ be such that $P_1\neq \pm P_2$, let $\mathcal{C}$ be the (unique) great circle that passes through $P_{1}$ and $P_{2}$,
and let $x\in \St$ be such that $\overline{x}\in \operatorname{\text{Coin}}(W_{P_1}, W_{P_2})$. Then either $x\in \operatorname{\text{Coin}}(U_{P_1}, U_{P_2})$ or $x\in \operatorname{\text{Coin}}(U_{P_1}, A\circ U_{P_2})$. Let $i\in \brak{1,2}$. Observe that any two great circles of $\St$ either coincide, or intersect in exactly two (antipodal) points, and that $U_{P_{i}}$ maps every great circle that passes through $P_{i}$ to itself. Suppose first that $x\notin \mathcal{C}$, and let $\mathcal{C}_{i}$ be the great circle that passes through $P_{i}$ and $x$. Then $\mathcal{C}_{1}\cap \mathcal{C}_{2}=\brak{x,A(x)}$. Since $U_{P_{i}}(x)\in \mathcal{C}_{i}$ and $x\in \operatorname{\text{Coin}}(U_{P_1}, U_{P_2}) \cup \operatorname{\text{Coin}}(U_{P_1}, A\circ U_{P_2})$, it follows that $U_{P_1}(x)\in \brak{x,A(x)}$, and so $p\circ U_{P_{1}}(x)=W_{P_{1}}(\overline{x})=\overline{x}\in \mathcal{C}$. By part~(\ref{it:prinWea}), this implies that $\overline{x}=\overline{P}_{1}$, which yields a contradiction since $x\notin \mathcal{C}$. So assume that $x\in \mathcal{C}$. We write the elements of $\mathcal{C}$ in exponential form, taking $P_{1}$ to be $e^{i\pi/2}$. Let $\rho$ be the oriented angle $\widehat{P_{1}P_{2}}$, and let $x=e^{i\varphi}$. Then $U_{P_{1}}(x)=e^{(2\varphi-\frac{\pi}{2})i}$ and $U_{P_{2}}(x)=e^{(2(\varphi-\rho)-\frac{\pi}{2})i+\rho i}= U_{P_{1}}(x) \ldotp e^{-\rho i}$. Since $x\in \operatorname{\text{Coin}}(U_{P_1}, U_{P_2}) \cup \operatorname{\text{Coin}}(U_{P_1}, A\circ U_{P_2})$, we see that $\rho\in \brak{0,\pi}$, but this implies that $P_{1}\in \brak{P_{2},-P_{2}}$, which yields a contradiction. We conclude that $W_{P_1}$ and $W_{P_2}$ are coincidence free.
\item Suppose on the contrary that the pair $(W_P, c_0)$ can be deformed to a pair of coincidence-free self-maps of $\rp$. By~\cite{Broo},
there exists a map $h\colon\thinspace \rp \to \rp$ that is homotopic to $W_{P}$ such that the pair $(h,c_{0})$ is coincidence free, and hence $h$ is non surjective. The maps $h$ and $c_{0}$ lift to maps $\widetilde{h}, \widetilde{c}_{0}\colon\thinspace \rp \to \St$, where $\widetilde{c}_{0}$ is also a constant map, and the non surjectivity of $h$ implies that of $\widetilde{h}$. Thus $\widetilde{h}$ is null homotopic, but then so is $h$, which yields a contradiction because $h$ is homotopic to $W_{P}$, and $W_{P}$ is non-null homotopic by part~(\ref{it:prinWeb}).\qedhere
\end{enumerate}
 \end{proof}
 
In the following proposition, we describe the set $[\rp, F_n(\rp)]$ (resp.\ $[\rp, D_n(\rp)]$) of homotopy classes of maps between $\rp$ and $F_{n}(\rp)$ (resp.\ $D_{n}(\rp)$),  and the set of homotopy classes of $n$-valued maps of $\rp$, from which we will see that they each contain two elements. 
Let $N=(0,0,1)\in \St$. As we mentioned in the proof of \relem{rp2split}, any $n$-valued map $\phi\colon\thinspace \rp \multimap \rp$ is split, and so the set of homotopy classes of $n$-valued maps of $\rp$ is equal to $\splitmap{\rp}{\rp}{n}/\!\sim$, where $\sim$ denotes the homotopy equivalence relation in $\splitmap{\rp}{\rp}{n}$.

 

\begin{prop}\label{prop:classmap} 
Let $n\geq 2$. 

\begin{enumerate}[(a)]
\item\label{it:classmap0} The sets $[\rp, F_n(\rp)]$, $[\rp, D_n(\rp)]$ and $\splitmap{\rp}{\rp}{n}/\!\sim$ are in bijection, and each possesses two elements.

\item\label{it:classmap1} The two homotopy classes of $n$-valued maps of $\rp$ may be described as follows:

\begin{enumerate}[(i)]
\item\label{it:classmapa} the first homotopy class consists of those $n$-valued maps $\phi\colon\thinspace \rp \multimap \rp$ such that any lift $\widehat{\Phi} \colon\thinspace \rp \to F_n(\rp)$ of $\phi$ induces the trivial homomorphism   
on the level of fundamental groups, 
and is homotopic to the constant map between $\rp$ and $F_{n}(\rp)$.
\item\label{it:classmapb} if $\phi_{n}\colon\thinspace \rp \multimap \rp$ is an $n$-valued map of $\rp$ that represents the second homotopy class, and $\widehat{\Phi}_n\colon\thinspace \rp \to F_n(\rp)$ is a lift of $\phi_n$, then for all $i=1,\ldots, n$, the composition of $\widehat{\Phi}_n$ with the projection $p_{i}\colon\thinspace F_n(\rp) \to \rp$
onto the $i\up{th}$ coordinate is homotopic to the map $W_{N}\colon\thinspace  \rp \to \rp$. Moreover, for all $i=1,\ldots, n+1$, the composition of $\widehat{\Phi}_{n+1}$ with the projection $q_{i}\colon\thinspace F_{n+1}(\rp) \to F_n(\rp)$ 
given by forgetting the $i\up{th}$ coordinate is homotopic to $\widehat{\Phi}_{n}$. 
\end{enumerate}
\end{enumerate}
\end{prop}

\begin{proof}
Let $n\geq 2$. 
\begin{enumerate}[(a)]
\item We start by showing that the set of homotopy classes $[\rp,F_{n}(\rp)]$ of $n$-ordered maps of $\rp$ has two elements. 
Consider the following Barratt-Puppe sequence:
\begin{equation}\label{eq:bps}
\ldots  \to [\St, F_n(\rp)] \to [\rp, F_n(\rp)] \to [\St[1], F_n(\rp)] \to [\St[1], F_n(\rp)]
\end{equation}
associated with the cofibration sequence $\St[1] \stackrel{2}\to \St[1] \to \rp \to \St \to \St \to \ldots$ for the space $F_n(\rp)$, 
 where the map $[\St[1], F_n(\rp)] \to [\St[1], F_n(\rp)]$ sends $[\beta]$ to $[\beta^2]$ for all maps $\beta\colon\thinspace \St[1] \to F_n(\rp)$. Now $[\St, F_n(\rp)]$ consists of a single homotopy class because $\pi_2(F_n(\rp))=0$~\cite[Corollary, p.~244]{FvB}. Therefore~\reqref{bps} implies that the set $[\rp, F_n(\rp)]$ is in one-to-one correspondence with the set of elements of $\pi_1(F_n(\rp))=P_n(\rp)$ of order less than or equal to $2$, namely the trivial element and the full twist $\ft$~\cite[Proposition~23]{GG3}. Let $\alpha\colon\thinspace \rp \to F_{n}(\rp)$ be an $n$-ordered map of $\rp$ whose homotopy class $[\alpha]$ corresponds to $\ft$. By~\reqref{bps}, the image of $[\alpha]$ in $[\St[1], F_n(\rp)]$ is non trivial, and so the induced homomorphism $\alpha_{\#}\colon\thinspace \pi_{1}(\rp)\to P_{n}(\rp)$ is non trivial (and injective). In particular, $\alpha$ is non-homotopic to the constant map $c\colon\thinspace \rp \to F_{n}(\rp)$, from which we conclude that $[\rp,F_{n}(\rp)]=\brak{[c],[\alpha]}$ has two distinct elements.   

We now prove that there are bijections between $[\rp,F_{n}(\rp)]$, $[\rp,D_{n}(\rp)]$ and $\splitmap{\rp}{\rp}{n}/\!\sim$.
As we pointed out in the proof of~\cite[Lemma~9(b)]{GG15}, the covering map $\pi\colon\thinspace F_n(\rp) \to D_n(\rp)$ induces a map $\widehat{\pi}:[\rp, F_n(\rp)] \to [\rp, D_n(\rp)]$ defined by $\widehat{\pi}([\Phi])=[\pi\circ \Phi]$ for all $\Phi\in F_{n}(\rp)^{\rp}$. If $\Psi\colon\thinspace \rp \to D_n(\rp)$ is a map and $\psi\colon\thinspace \rp \multimap \rp$ is the associated $n$-valued map, then $\psi$ is split by~\cite[Lemma~14]{GG15}. It follows that there exists a lift $\widehat{\Psi}\colon\thinspace \rp \to F_{n}(\rp)$ of $\psi$, and this map satisfies $\pi\circ \widehat{\Psi}=\Psi$. In particular, $\widehat{\pi}$ is surjective, and hence $[\rp,D_{n}(\rp)]=\brak{[\pi\circ c],[\pi\circ \alpha]}$ has at most two elements. Since $\pi_{\#}\colon\thinspace P_{n}(\rp)\to B_{n}(\rp)$ is inclusion, $(\pi\circ \alpha)_{\#} \colon\thinspace \pi_{1}(\rp) \to B_{n}(\rp)$ is injective, and thus $[\pi\circ c]\neq [\pi\circ \alpha]$. Therefore $\widehat{\pi}$ is a bijection.

Finally, by~\cite[Lemma~9(b)]{GG15}, the set $\splitmap{\rp}{\rp}{n}/\!\sim$ is in one-to-one correspondence with the orbits of the set $[\rp, F_n(\rp)]$ under the action of $S_{n}$ induced by that of $S_{n}$ on $F_{n}(\rp)^{\rp}$. But $[\rp,F_{n}(\rp)]=\brak{[c],[\alpha]}$, and since the orbit of the homotopy class $[c]$ of the constant map under the action of $S_{n}$ must be $\brak{[c]}$, the orbit of $[\alpha]$ must be $\brak{[\alpha]}$. It follows that $\splitmap{\rp}{\rp}{n}/\!\sim$ has precisely two elements.

\item\begin{enumerate}[(i)]
\item If $\phi\colon\thinspace \rp \multimap \rp$ is an $n$-valued map of $\rp$ such that the associated map $\Phi\colon\thinspace \rp \to D_{n}(\rp)$ belongs to the homotopy class $[\pi\circ c]$ in $[\rp,D_{n}(\rp)]$ of the constant map $c\colon\thinspace \rp \to F_{n}(\rp)$, then it follows from the proof of part~(\ref{it:classmap0}) that any lift $\widehat{\Phi}\colon\thinspace \rp \to F_n(\rp)$ of $\phi$ is homotopic to the constant map, and induces the trivial homomorphism on the level of fundamental groups.

\item Let $\phi_{n}\colon\thinspace \rp \multimap \rp$ be an $n$-valued map of $\rp$ that represents the second (non-trivial) homotopy class, and let $\widehat{\Phi}_n\colon\thinspace \rp \to F_n(\rp)$ be a lift of $\phi_n$, so that $\widehat{\Phi}_n$ is homotopic to the map $\alpha$ given in the proof of part~(\ref{it:classmap0}). For all $1\leq i \leq n$, the image of $\ft$ under the homomorphism induced by the projection $p_{i}\colon\thinspace F_n(\rp) \to \rp$ is trivial (this follows from the proof of~\cite[Proposition~8]{GG14} using the fact that $\ft$ may be written as a product of the generators $(A_{i,j})_{1\leq i,j\leq n}$ given in that paper). Since the image of the induced homomorphism $\alpha_{\#}\colon\thinspace \pi_{1}(\rp) \to P_n(\rp)$ is equal to $\ang{\ft}$, the composition $p_{i}\circ \widehat{\Phi}_{n}\colon\thinspace \rp \to \rp$ induces the trivial homomorphism on the level of fundamental groups. Now let $P_1,\ldots,P_n$ be $n$ distinct points of $\St$ that lie on the geodesic arc between $(1,0,0)$ and $(0,0,1)$, and consider the map $(W_{P_1},\ldots, W_{P_n})\colon\thinspace  \rp \to F_n(\rp)$. The fact that this map is well defined is a consequence of \relem{prinWe}(\ref{it:prinWed}). Since $p_i\circ (W_{P_1},\ldots, W_{P_n})=W_{P_i}$ for all $i=1,\ldots,n$, the map $(W_{P_1},\ldots, W_{P_n})$ is not homotopic to the constant map $c$, and so it is homotopic to $\alpha$ by the proof of part~(\ref{it:classmap0}). In particular, $\widehat{\Phi}_n$ is homotopic to $(W_{P_1},\ldots, W_{P_n})$, and the statements of part~(\ref{it:classmap1})(\ref{it:classmapb}) then follow.
\qedhere
%
\end{enumerate}
\end{enumerate}
 \end{proof} 
 
 We are now able to prove the Wecken property for $\rp$.


\begin{proof}[Proof of \reth{wecproj}]
As we mentioned previously, $\rp$ has the Wecken property for self-maps. So suppose that $n>1$. From \repr{classmap}, there are two homotopy classes of $n$-valued maps of $\rp$, and so by \relem{rp2split},
it suffices to show that each of these classes admits a representative for which the number of fixed points is equal to $n$.
%
Let $P_1,\ldots,P_n$ be as in the proof of \repr{classmap}(\ref{it:classmap1})(\ref{it:classmapb}), and let $c_{P_{i}}\colon\thinspace \rp \to \rp$ be the constant map at $P_{i}$. From \repr{classmap}, the two homotopy classes of $n$-valued maps of $\rp$ contain an $n$-valued map $\phi\colon\thinspace \rp \multimap \rp$ that is split and admits a lift $\widehat{\Phi}=(\phi_{1},\ldots,\phi_{n})\colon\thinspace \rp \to F_{n}(\rp)$, where either $\phi_{i}=c_{\overline{P_{i}}}$ is the constant map at $\overline{P_{i}}$ for all $i=1,\ldots,n$, or $\phi_{i}=W_{P_{i}}$ for all $i=1,\ldots,n$. Using \relem{prinWe}(\ref{it:prinWea}), $\operatorname{\text{Fix}}(\phi_{i})=\brak{\overline{P_{i}}}$ for all $i=1,\ldots,n$, and hence $\phi$ has exactly $n$ fixed points.
Moreover, by \relem{rp2split} we have 
$N(\phi)=n$.
So each of the two homotopy classes of $n$-valued maps contains a representative $\phi$ for which $N(\phi)=n$, and hence $\rp$ has the Wecken property for $n$-valued maps.
\end{proof}

\section{Nielsen numbers of $n$-valued maps}\label{sec:calcu}

The Nielsen number for $n$-valued maps of a compact polyhedron $X$ was defined in~\cite{Sch1}, and may be determined for split $n$-valued maps using \reth{helgath0}. The aim of this section is to prove \reth{form}, where we give a formula for the Nielsen number of non-split $n$-valued maps of a space $X$
that is in the same spirit as that of \reth{helgath0}.
For 
\reth{form}, we shall require $X$ to be a compact, orientable manifold without boundary, in order to have the notions of index, Lefschetz and Nielsen numbers for coincidences of pairs of maps from finite coverings of $X$ to $X$ at our disposal. However, many of the results that lead to \reth{form} are valid under weaker hypotheses on $X$, namely those of 
\repr{nielsen0} below.

We start by recalling some notation and results from~\cite[Section~3.2]{GG15} that will be used throughout the rest of the paper. Given an $n$-valued map $\phi\colon\thinspace  X \multimap X$ of a topological space $X$ that is locally path-connected and semi-locally simply connected, we consider the corresponding map $\Phi\colon\thinspace X \to D_n(X)$, and the induced homomorphism $\Phi_{\#}\colon\thinspace \pi_1(X) \to \pi_1(D_n(X))$ on the level of fundamental groups, where $\pi_1(D_n(X))=B_n(X)$. By the short exact sequence
\begin{equation*}
1\to P_n(X) \to B_n(X) \stackrel{\tau}{\to} S_n \to 1,
\end{equation*}
$P_n(X)$ is a normal subgroup of $B_n(X)$ of finite index $n!$, so the subgroup $H=\Phi_{\#}^{-1}(P_n(X))$ is a normal subgroup of $\pi_1(X)$ of finite index. Let $L$ be the finite quotient group $\pi_{1}(X)/H$, and let $q\colon\thinspace \widehat{X} \to X$ be the covering of $X$ that corresponds to the subgroup $H$. Such a covering exists due to the hypotheses on $X$. As the following proposition shows, the fixed points of $\phi$ may be described in terms of the coincidences of $q$ with the coordinate maps $f_{1},\ldots, f_{n}\colon\thinspace \widehat{X}\to X$ of a lift of the $n$-valued map $\phi\circ q \colon\thinspace \widehat{X}\multimap X$.

\begin{prop}\cite[Propositions~16 and~17]{GG15}\label{prop:nielsen0}
Let $n\in \N$, and suppose that $X$ is a connected, locally arcwise-connected metric space. 
\begin{enumerate}[(a)]
\item\label{it:nielsen} With the above notation, the $n$-valued map $\phi_1=\phi \circ q\colon\thinspace  \widehat{X}\multimap X$ admits exactly $n!$ lifts, which are  $n$-ordered  maps from $\widehat{X}$ to $F_n(X)$. If one such lift $\widehat{\Phi}_{1}\colon\thinspace \widehat{X}\to F_n(X)$ is given by $\widehat{\Phi}_{1}=(f_1,\ldots, f_n)$, where for $i=1,\ldots,n$, $f_i$ is a map from $\widehat{X}$ to $X$, then the other lifts are of the form $(f_{\rho(1)},\ldots,f_{\rho(n)})$, where $\rho\in S_n$.  
\item\label{it:coinfix} if the lift $\widehat{\Phi}_{1}=(f_1,\ldots, f_n)$ is as in part~(\ref{it:nielsen}) then 
the restriction of $q\colon\thinspace \widehat{X} \to X$ to $\bigcup_{i=1}^{n} \operatorname{\text{Coin}}(q, f_i) \to \operatorname{\text{Fix}}(\phi)$ is surjective. Furthermore, the pre-image of a point $x\in \operatorname{\text{Fix}}(\phi)$ by this map is precisely $q^{-1}(x)$, namely the fibre over $x\in X$ of the covering map $q$.
\end{enumerate}
\end{prop} 

Although the lift $\widehat{\Phi}_1$ is not unique, the set $\brak{f_{1},\ldots,f_{n}}$ is, and so the set $\bigcup_{i=1}^{n} \operatorname{\text{Coin}}(q, f_i)$ is independent of the choice of lift of $\phi$. In what follows, we aim to describe the Nielsen classes of $\phi$ in terms of the Nielsen coincidence classes of the pairs $(q, f_i)$, where $i=1,\ldots,n$, which will lead to a formula for $N(\phi)$ similar in spirit to that of \reth{helgath0}.
Observe that the composition $\pi_1(X) \stackrel{\Phi_{\#}}{\to} B_n(X) \stackrel{\tau}{\to} S_n$ 
is a homomorphism whose kernel is $H$,
so it induces an injective homomorphism $\Gamma\colon\thinspace L \to S_n$. Let $L'=\im{\Gamma}=\im{\tau \circ \Phi_{\#}}$, and for $i,j=1,\ldots,n$, let $L_{i,j}'=\setr{\rho\in L'}{\rho(i)=j}$. The subset $L_{i,i}'$ is a subgroup of $L'$, and if $i,j\in \brak{1,\ldots,n}$, the subset $L'_{i,j}$ is either empty or is a left coset of $L'_{i,i}$ in $L'$.

In the rest of this section, we will suppose without further comment that $X$ satisfies the hypotheses of \repr{nielsen0}, so that it is a connected, locally arcwise-connected metric space. If $\phi\colon\thinspace X \multimap X$ is an $n$-valued map, 
we recall the Nielsen relation on $\operatorname{\text{Fix}}(\phi)$, the index of a Nielsen fixed point class of $\phi$ and the definition of the Nielsen number $N(\phi)$ of $\phi$ from~\cite[Section~5]{Sch1}. For the definition of index, using~\cite[Theorem~6]{Sch0}  and the homotopy invariance of the Nielsen number~\cite[Theorem 6.5]{Sch0}, without loss of generality, we may restrict ourselves to the case where $\operatorname{\text{Fix}}(\phi)$ is finite. First note that by~\cite[Lemma~12]{GG15}, if $\lambda\colon\thinspace I \to X$ is a path then the $n$-valued map $\phi \circ \lambda\colon\thinspace I \multimap X$ is split. Let $x,x'\in \operatorname{\text{Fix}}(\phi)$. We say that $x$ and $x'$ are Nielsen equivalent if there exist maps $g_1,g_2,\ldots,g_n \colon\thinspace I \to X$, a path $\lambda\colon\thinspace I \to X$ from $\lambda(0)=x$ to $\lambda(1)=x'$ and $j\in \brak{1,\ldots, n}$ such that $\phi \circ \lambda =\brak{g_1,g_2,\ldots,g_n}$, and $g_{j}$ is a path from $g_j(0)=x$ to $g_j(1)=x'$ that is homotopic to $\lambda$ relative to their endpoints. This defines an equivalence relation on $\operatorname{\text{Fix}}(\phi)$, and the resulting equivalence classes are called \emph{Nielsen fixed point classes} of the $n$-valued map $\phi$. 

To define the index of an isolated point $x$ in $\operatorname{\text{Fix}}(\phi)$,
we suppose that $X$ is a compact polyhedron. Following~\cite[Section~3]{Sch1},
let $x$ be in the interior of a maximal simplex $\overline{\sigma}$. By~\cite[Splitting~Lemma~2.1]{Sch1}
$\phi\left\lvert_{\overline{\sigma}}\right.$ is split, so may be written in the form
 $\phi\left\lvert_{\overline{\sigma}}\right.=\brak{f_1,\ldots,f_n}$,
where $x \in \operatorname{\text{Fix}}(f_j)$ for some (unique) $1\leq j\leq n$. We then define $\operatorname{\text{Ind}}(\phi, x)=\operatorname{\text{Ind}}(f_j, x)$, where the right-hand side is the usual fixed point index (see~\cite[Sections~3 and~5]{Sch1} for more details).  As in the single-valued case, 
the \emph{index} of a Nielsen fixed point class of $\phi$ that contains a  finite number of points is the sum of the indices of these fixed points, and such a fixed point class is said to be \emph{essential} if its index is non zero. The \emph{Nielsen number} $N(\phi)$ is defined to be the number of essential Nielsen fixed point classes of $\phi$.


\begin{rem}\label{rem:splitpath}
Within our framework, the maps $g_1,g_2,\ldots,g_n$ may be chosen as follows:  given $x_0,x_0'\in \operatorname{\text{Fix}}(\phi)$, a point $\widetilde{x}_{0}\in \widehat{X}$ such that $q(\widetilde{x}_{0})=x_{0}$, and a path $\lambda\colon\thinspace I \to X$ from $x_0$ to $x_0'$,
let $\widetilde{\lambda}\colon\thinspace I \to \widehat{X}$ be the unique lift of $\lambda$ to $\widehat{X}$ for which $\widetilde{\lambda}(0)=\widetilde{x}_{0}$.  Consider the $n$-ordered map $\widehat{\Phi}_1=(f_1,\ldots,f_n)\colon\thinspace \widehat{X} \to F_{n}(X)$ given by \repr{nielsen0}(\ref{it:nielsen}). Then $\phi \circ \lambda=\brak{g_1,\ldots,g_n}$, where for $i=1,\ldots,n$, $g_{i}\colon\thinspace I \to X$ is the map defined by
$g_i=f_i\circ \widetilde \lambda$. 
So  $x_0$ and $x_0'$ are Nielsen equivalent if there is a path $\lambda$ as above and $j\in \brak{1,\ldots,n}$ such that 
$\lambda(0)=g_j(0)=f_j\circ \widetilde \lambda(0)$, $\lambda(1)=g_j(1)=f_j\circ \widetilde \lambda(1)$ and $\lambda$ is homotopic to $g_j$ relative to their endpoints.
\end{rem}

In the following lemmas, we will compare the coincidences of $q$ and the $f_i$ with the fixed points of $\phi$.

\begin{lem}\label{lem:auxil0}  
With the above notation, let $x_0\in \operatorname{\text{Fix}}(\phi)$, let $\widetilde{x}_{0}\in \widehat{X}$ and $i\in \brak{1,\ldots,n}$ be such that $q(\widetilde{x}_{0})=x_0$ and $\widetilde{x}_{0}\in \operatorname{\text{Coin}}(q,f_i)$. If $y\in q^{-1}(x_0)$ and $j\in \brak{1,\ldots,n}$
then $y \in \operatorname{\text{Coin}}(q, f_j)$ if and only if $\Gamma(\alpha)\in L'_{i,j}$, where $\alpha$ is the $H$-coset of $[q(\gamma)]$, $\gamma$ being any path from $\widetilde{x}_{0}$ to $y$. In particular, the points of $q^{-1}(x_0)$ that belong to $ \operatorname{\text{Coin}}(q,f_i)$ are in one-to-one  correspondence with 
the subgroup $L'_{i,i}$. 
\end{lem}

Since $L_{i,i}'$ does not depend on $x_0$, \relem{auxil0} implies that for all $x \in \operatorname{\text{Fix}}(\phi)$, the set $q^{-1}(x)\cap \operatorname{\text{Coin}}(q,f_i)$ has $\lvert L'_{i,i}\rvert$ elements.
 
\begin{proof}[Proof of \relem{auxil0}]
The proof makes use of basic covering space theory. Let $\alpha\in L$ be the unique deck transformation for which $y$ is equal to the element $\alpha \ldotp \widetilde{x}_{0}$ of $\widehat{X}$ that arises from the action of deck transformation group $L$ on $\widehat{X}$. Then $\Gamma(\alpha)\in L'$ defines a deck transformation of the covering  $\pi\colon\thinspace F_n(X)\to D_n(X)$. Using the fact that $q$ and $\pi$ are covering maps and $\widehat{\Phi}_1$ is a lift of $\Phi$, we have:
\begin{equation}\label{eq:phigamma}
\widehat{\Phi}_1(\alpha \ldotp \widetilde{x})= \Gamma(\alpha) \ldotp \widehat{\Phi}_1(\widetilde{x})
\end{equation}
for all $\widetilde{x}\in q^{-1}(x_0)$, where $\Gamma(\alpha) \ldotp \widehat{\Phi}_1(\widetilde{x})$ is the element of $F_{n}(X)$ arising from the action of $S_{n}$ on $F_{n}(X)$. Since $\widetilde{x}_{0}\in \operatorname{\text{Coin}}(q,f_i)$, $\widehat{\Phi}_1(\widetilde{x}_0)$ is an element of $F_n(X)$ whose $i\up{th}$ coordinate is $x_0$, and $\Gamma(\alpha) \ldotp \widehat{\Phi}_1(\widetilde{x}_0)$ is an element of $F_n(X)$ whose $\Gamma(\alpha)(i)\up{th}$ coordinate is $x_0$. So if $y\in \operatorname{\text{Coin}}(q,f_j)$, the $j\up{th}$  coordinate of $\widehat{\Phi}_1(y)$ is $x_0$, and thus $\Gamma(\alpha) \in L_{i,j}'$. Conversely, if $\Gamma(\alpha) \in L_{i,j}'$ then $\Gamma(\alpha)(i)=j$, so the $j\up{th}$  coordinate of $\widehat{\Phi}_1(y)$ is $x_0$, and hence $y\in \operatorname{\text{Coin}}(q,f_j)$.
For the last part of the statement, since $L_{i,i}'$ is a subgroup of $L'$, it suffices to take $j=i$.
\end{proof}   
   
With \relem{auxil0} in mind, we define the following notation. 
For each  $i\in \brak{1,\ldots,n}$, let $\mathbb{O}_i$ be the orbit of $i$ by the action of the subgroup $L'$ of $S_{n}$ on the set $\brak{1,\ldots,n}$, and let $I_0=\brak{i_1,\ldots,i_s}$ be such that the sets $\brak{\mathbb{O}_i}_{i\in I_{0}}$ form a partition of $\brak{1,\ldots,n}$.
As examples, if $H=\pi_1(X)$  (resp.\ $L'=S_n$) then $\mathbb{O}_i=\brak{i}$ (resp.\ $\mathbb{O}_i=S_n$) for all $i\in \brak{1,\ldots,n}$. The following result underlines the relevance of these orbits. 

\begin{lem}\label{lem:auxil}
With the above notation, let $x_0\in \operatorname{\text{Fix}}(\phi)$. Then there exists $i\in \brak{1,\ldots,n}$ such that: 
\begin{equation*}
\setr{j\in \brak{1,\ldots,n}}{q^{-1}(x_0)\cap \operatorname{\text{Coin}}(q, f_j)\neq \vide}=\mathbb{O}_i.
\end{equation*}
\end{lem}
 
\begin{proof}
The proof uses arguments similar to those of \relem{auxil0}. Let $x_0\in \operatorname{\text{Fix}}(\phi)$, and let $\widetilde{x}_{0}\in \widehat{X}$ be a lift of $x_{0}$. By \repr{nielsen0}(\ref{it:coinfix}), 
$\widetilde{x}_{0}$ belongs to $\operatorname{\text{Coin}}(q,f_i)$ for some $i\in \brak{1,\ldots,n}$. 
First, suppose that $y\in q^{-1}(x_0)\cap \operatorname{\text{Coin}}(q, f_j)$ for some $j\in \brak{1,\ldots,n}$, and let $\alpha\in L$ be such that $\alpha\ldotp \widetilde{x}_{0}=y$. From~\reqref{phigamma}, $\Gamma(\alpha)(i)=j$, so $j\in \mathbb{O}_i$. Conversely, suppose that $j\in \mathbb{O}_i$. Then there exists $\alpha\in L$ such that $\Gamma(\alpha)(i)=j$, and taking $y=\alpha\ldotp \widetilde{x}_{0}$, we have $y\in q^{-1}(x_0)\cap \operatorname{\text{Coin}}(q, f_j)$.
\end{proof}

Note that by \repr{nielsen0}(\ref{it:coinfix}) and \relem{auxil}, $\operatorname{\text{Fix}}(\phi)= q\bigl( \bigcup_{j\in I_{0}} \operatorname{\text{Coin}}(q,f_j)\bigr)$. Since we wish to express the Nielsen number of $\phi$ in terms of the Nielsen coincidence numbers of the pairs $(q, f_i)$, for the values of $i$ belonging to $I_0$, 
we shall compare the Nielsen coincidence relation and the Nielsen coincidence number of the pairs $(q, f_i)$ with the Nielsen relation and the Nielsen number for $\phi$ respectively.

\begin{lem}\label{lem:nielsenclasses}
With the above notation, let $i\in \brak{1,\ldots,n}$, and let $y_1$ and $y_2$ be elements 
of $\operatorname{\text{Coin}}(q, f_i)$ that belong to the same Nielsen coincidence class of the pair $(q,f_{i})$. Then $q(y_1)$ and $q(y_2)$ are elements of $\operatorname{\text{Fix}}(\phi)$ that belong to the same Nielsen fixed point class of the $n$-valued map $\phi$. Further, $q$ sends each Nielsen coincidence class of the pair $(q,f_i)$ surjectively onto a Nielsen fixed point class of $\phi$.
\end{lem}

\begin{proof} 
Let $y_1$ and $y_2$ be elements of $\operatorname{\text{Coin}}(q, f_i)$ for some $i\in \brak{1,\ldots,n}$ that belong to the same Nielsen coincidence class of the pair $(q,f_{i})$. By \repr{nielsen0}(\ref{it:coinfix}), $q(y_1)$ and $q(y_2)$ are fixed points of $\phi$. Since $y_1$ and $y_2$ belong to the same Nielsen coincidence class of $(q,f_{i})$, there exists a path $\widetilde{\lambda}\colon\thinspace I \to \widehat{X}$ from $y_{1}$ to $y_{2}$ such that the path $f_{i}\circ \widetilde{\lambda}$ is homotopic in $X$ to $\lambda$ relative to the endpoints $q(y_1)$ and $q(y_2)$, where $\lambda \colon\thinspace I \to X$ is the path defined by $\lambda=q\circ \widetilde{\lambda}$. Now the $n$-valued map $\phi \circ \lambda\colon\thinspace I \multimap X$ is split by~\cite[Lemma~12]{GG15}, and by \rerem{splitpath}, $\phi \circ \lambda=\brak{g_{1},\ldots,g_{n}}$, where for $j=1,\ldots,n$, $g_{j}= f_{j} \circ \widetilde{\lambda}$. So $g_{i}(0)= f_{i} \circ \widetilde{\lambda}(0)=f_{i}(y_{1})=q(y_{1})$, $g_{i}(1)= f_{i} \circ \widetilde{\lambda}(1)=f_{i}(y_{2})=q(y_{2})$, and $g_{i}$ is homotopic to $\lambda$ relative to the endpoints $q(y_1)$ and $q(y_2)$, from which we deduce that $q(y_1)$ and $q(y_2)$ belong to the same Nielsen fixed point class of $\phi$. To prove the second part of the statement, by the first part, it suffices to show that if $x$ is a fixed point of $\phi$ that belongs to the same Nielsen fixed point class of $\phi$ as $q(y_{1})$ then there exists $y\in \operatorname{\text{Coin}}(q,f_i)$ such that $q(y)=x$, and $y$ and $y_{1}$ belong to the same Nielsen coincidence class of the pair $(q,f_{i})$. To see this, note that by \rerem{splitpath}, there exist a path $\lambda\colon\thinspace I \to X$ from $q(y_{1})$ to $x$, a lift $\widetilde{\lambda}\colon\thinspace I \to \widehat{X}$ such that $\widetilde{\lambda}(0)=y_{1}$ and $j\in \brak{1,\ldots,n}$ such that $\lambda(0)=g_j(0)=f_j\circ \widetilde \lambda(0)$, $\lambda(1)=g_j(1)=f_j\circ \widetilde{\lambda}(1)$, and $\lambda$ is homotopic to $g_j$ relative to their endpoints. In particular, $f_{j}(y_{1})= q(y_{1})$, and so $j=i$ because $y_{1}\in \operatorname{\text{Coin}}(q, f_i)$. Further, if $y=\lambda(1)$ then $q(y)=x$ because $\widetilde{\lambda}$ is a lift of $\lambda$, and $x=\lambda(1)=f_{i}(y)$, so $y\in \operatorname{\text{Coin}}(q, f_i)$. Finally, the paths $\lambda$ and $g_{i}=f_i\circ \widetilde{\lambda}$ are homotopic in $X$ relative to their endpoints, and hence $y_{1}$ and $y$ belong to the same Nielsen coincidence class of the pair $(q,f_{i})$ as required.
\end{proof}

In order to obtain a formula for $N(\phi)$, another ingredient that we require is the number of points of $q^{-1}(x_{0})\cap \operatorname{\text{Coin}}(q,f_i)$ 
that belong to the same Nielsen coincidence class of the pair $(q,f_i)$, where $x_0\in \operatorname{\text{Fix}}(\phi)$ and $i\in I_0$. 
Suppose that $q^{-1}(x_{0})\cap \operatorname{\text{Coin}}(q,f_i)\neq \vide$, and let 
$\widetilde{x}_0,y\in\widehat{X}$ be elements of this intersection.
There exists a unique $\mu \in L=\pi_1(X)/H$ such that $y=\mu \ldotp \widetilde{x}_0$. Let $L_i$ be the subset of $L$ consisting of such $\mu$ as $y$ runs over the elements of $q^{-1}(x_0)\cap \operatorname{\text{Coin}}(q,f_i)$. 
Note that $L_i=\set{\mu \in L}{\Gamma(\mu)(i)=i}$, in particular, $L_{i}$ is independent of $\widetilde{x}_0$, $L_i$ is a subgroup of $L$, the order of $L_i$ is equal to the cardinality of $q^{-1}(x_0)\cap{\text{Coin}}(q,f_i)$, and $L_{i,i}'=\Gamma(L_i)$, so $L_{i,i}' \cong L_i$.
If $\mu\in L_i$, consider the corresponding element $y\in q^{-1}(x_0)\cap\operatorname{\text{Coin}}(q,f_i)$ defined by $y=\mu \ldotp \widetilde{x}_0$, and let $\gamma \colon\thinspace I \to \widehat{X}$ be a path from $\widetilde{x}_0$ to $y$. Then $f_i\circ \gamma$ and $q\circ \gamma$ are loops in $X$ based at $x_{0}$. Let $W_{\widetilde x_0}(\mu)$ be the subset of $\pi_1(X)$ of loop classes of the form $[(q\circ \gamma) \ast (f_i\circ \gamma)^{-1}]$, where $\gamma$ runs over the set of paths from $\widetilde{x}_0$ to $y$. Observe that  $W_{\widetilde x_0}(\mu)$  contains the trivial element of $\pi_1(X)$ if and only if $\widetilde{x}_0$ and $y$ belong to the same Nielsen coincidence class for the pair $(q,f_i)$.
With this in mind, let $K_i(\widetilde{x}_0)$ be the set of elements $y\in q^{-1}(x_0)\cap\operatorname{\text{Coin}}(q,f_i)$    
for which $W_{\widetilde x_0}(\mu)$ contains the trivial element, where $\mu\in L_i$ is such that $y=\mu \ldotp \widetilde{x}_0$. Then $K_i(\widetilde{x}_0)$ is the subset of elements $q^{-1}(x_0)\cap\operatorname{\text{Coin}}(q,f_i)$ that belong to the same Nielsen coincidence class of the pair $(q,f_i)$ as $\widetilde{x}_0$. Let $\lvert K_i(\widetilde{x}_0)\rvert$ denote the cardinality of $K_i(\widetilde{x}_0)$.

\begin{lem}\label{lem:fund1} 
With the above notation, let $x_{0}\in \operatorname{\text{Fix}}(\phi)$, and let $\widetilde{x}_0\in q^{-1}(x_{0})\cap \operatorname{\text{Coin}}(q,f_i)$.\vspace*{-1mm}

\begin{enumerate}
\item\label{it:fund1a} The number of coincidence points of the pair $(q, f_i)$ that belong to $q^{-1}(x_{0})$ is equal to the order of the subgroup $L_i$.
\item\label{it:fund1b} If $\widetilde{z}\in q^{-1}(x_0)\cap  \operatorname{\text{Coin}}(q,f_i)$, the set $K_i(\widetilde{z})$ of points of $q^{-1}(x_0)$ that belong to the same Nielsen coincidence class for the pair $(q,f_{i})$ as $\widetilde{z}$ is in one-to-one correspondence with the set $K_i(\widetilde{x}_0)$.
\item\label{it:fund1c} Let  $\widetilde{z}\in  \operatorname{\text{Coin}}(q,f_i)$ be such that $\widetilde{x}_0$ and $\widetilde{z}$ belong to the same Nielsen coincidence class for the pair $(q,f_i)$. Then $K_i(\widetilde{z})$ is in one-to-one correspondence with the set $K_i(\widetilde{x}_0)$.
\end{enumerate}
\end{lem} 

\begin{proof}\mbox{}
\begin{enumerate}
\item This follows from \relem{auxil0} and the isomorphism $L_i \cong L_{i,i}'$. 
\item Let us construct an injective map from $K_i(\widetilde{x}_0)$ to 
$K_i(\widetilde{z})$.
Given $y\in K_i(\widetilde{x}_0)$,
by definition, there exists a path $\gamma\colon \thinspace I \to \widehat{X}$ such that $\gamma(0)=\widetilde{x}_0$, $\gamma(1)=\widetilde{y}$  and $q\circ \gamma$ is homotopic to $f_i\circ \gamma$ relative to their endpoints. Let $\gamma_1$ be the unique lift of $q\circ \gamma$ for which $\gamma_1(0)=\widetilde{z}$. Since $q\circ \gamma$ is a loop in $X$ based at $x_{0}$, $\gamma_1(1)\in q^{-1}(x_{0})$. 
We claim that $\gamma_1(1)\in  K_i(\widetilde{z})$. To see this, let us show that $q\circ \gamma_1=q\circ \gamma$ is homotopic to $f_i \circ \gamma_1$ relative to their endpoints. It suffices to observe that $f_i\circ \gamma=f_i\circ \gamma_1$. Since the two compositions $\widehat{X} \stackrel{q}{\to} X \stackrel{\Phi}{\to} D_{n}(X)$ and $\widehat{X} \stackrel{\widehat{\Phi}_{1}}{\to} F_{n}(X) \stackrel{\pi}{\to} D_{n}(X)$ are equal, we have $\pi(\widehat \Phi\circ \gamma_1)=\pi(\widehat \Phi\circ \gamma)$, in other words, the $n$-tuple $(f_{1}\circ \gamma, \ldots, f_{n}\circ \gamma)$ of paths is a permutation of the $n$-tuple $(f_{1}\circ \gamma_{1}, \ldots, f_{n}\circ \gamma_{1})$ of paths. Since these two $n$-tuples are paths in $F_{n}(X)$ and $f_{i}\circ \gamma(0)=f_{i}\circ \gamma_{1}(0)=x_{0}$, it follows that $f_{i}\circ \gamma=f_{i}\circ \gamma_{1}$, and hence $f_{i}\circ \gamma_{1}$ is homotopic to $q\circ \gamma_1$ relative to their endpoints. In particular, $f_{i}\circ \gamma_{1}(1)=f_{i}\circ \gamma(1)=q\circ \gamma(1)=q\circ \gamma_{1}(1)$, so $\gamma_{1}(1)\in q^{-1}(x_{0}) \cap \operatorname{\text{Coin}}(q,f_i)$, and thus $\gamma_{1}(1)\in K_{i}(\widetilde{z})$, which proves the claim. By construction, the map from $K_i(\widetilde{x}_0)$ to $K_{i}(\widetilde{z})$ that to $y$ associates $\gamma_{1}(1)$ is injective. Exchanging the r\^{o}les of $\widetilde{x}_0$ and $z$, we see that $K_i(\widetilde{x}_0)$ and $K_{i}(\widetilde{z})$ have the same number of elements, and that this map is a bijection.
\item Let $\lambda\colon\thinspace I \to \widehat{X}$ be a path such that $\lambda(0)=\widetilde{x}_0$,  $\lambda(1)=\widetilde{z}$ and $q\circ \lambda$ is homotopic to $f_i\circ \lambda$ relative to their endpoints. Let $y\in q^{-1}(x_{0})\cap \operatorname{\text{Coin}}(q,f_i)$ be a point that belongs to the same Nielsen coincidence class of $(q,f_i)$ as $\widetilde{x}_0$, and let $\gamma\colon\thinspace I \to \widehat{X}$ be a path such that $\gamma(0)=\widehat{x}_0$, $\gamma(1)=y$ and $q\circ \gamma$ is homotopic to $f_i\circ \gamma$ relative to their endpoints. Let $\lambda'\colon\thinspace I \to \widehat{X}$ be the unique lift of $q\circ \lambda$ for which $\lambda'(0)=y$, and let $\gamma'\colon\thinspace I \to \widehat{X}$ be the path defined by $\gamma'=\lambda^{-1}\ast\gamma\ast\lambda'$. Then $\gamma'(0)=\widetilde{z}$, $q\circ \gamma'(1)=q\circ \lambda'(1)=q\circ \lambda(1)=q(\widetilde{z})$, so $\gamma'(1)\in q^{-1}(q(\widetilde{z}))$. As in the proof of part~(\ref{it:fund1b}), $q\circ \lambda'$ is homotopic to $f_i\circ \lambda'$ relative to their endpoints, and it follows that $q\circ \gamma'$ is homotopic to $f_i\circ \gamma'$ relative to their endpoints, so $\gamma'(1)\in K_i(\widetilde{z})$. We thus obtain an injective map from $K_i(\widetilde{x}_0)$ and $K_i(\widetilde{z})$ that to $y$ associates $\gamma'(1)$. By exchanging the r\^{o}les of $\widetilde{x}_0$ and $\widetilde{z}$, we see that this map is a bijection.\qedhere
\end{enumerate}
\end{proof}

If $i\in \brak{1,\ldots,n}$, by \relem{fund1}, the set $q^{-1}(x_0)\cap  \operatorname{\text{Coin}}(q,f_i)$ contains $\lvert L_i\rvert$ points, and is  partitioned into Nielsen coincidence classes of the pair $(q,f_i)$ that each contain $\lvert K_i(\widetilde{x}_0)\rvert$ points. So this set is partitioned into $\lvert L_i\rvert/\lvert K_i(\widetilde{x}_0)\rvert$ disjoint subsets each of which is contained in a Nielsen coincidence class of $(q,f_i)$. If $W_1, W_2$ are Nielsen coincidence classes of $(q,f_i)$, we say that they are \emph{related} if $q(W_1)\cap q(W_2)\neq \vide$. This defines an equivalence relation. Let $C_1, \ldots,C_r$ denote the corresponding set of equivalence classes of the Nielsen coincidence classes of $(q,f_i)$.

\begin{lem}\label{lem:auxiV} 
With the above notation, for $j=0,1$, let $x_j \in \operatorname{\text{Fix}}(\phi)$, and suppose that $\widetilde{x}_j \in q^{-1}(x_j)\cap \operatorname{\text{Coin}}(q,f_i)$ for some $i\in \brak{1,\ldots,n}$. If $\widetilde x_0$ and $\widetilde{x}_1$ belong to related Nielsen coincidence classes of $(q,f_i)$ then $\lvert K_i(\widetilde{x}_0)\rvert=\lvert K_i(\widetilde{x}_1)\rvert$.
\end{lem}

\begin{proof}
We shall construct a bijection between  $ K_i(\widetilde{x}_0)$ and $ K_i(\widetilde{x}_1)$. First, we claim that there is an element $\widetilde{x}_1' \in q^{-1}(q(\widetilde{x}_1))$ that belongs to the same Nielsen coincidence class of $(q,f_i)$ as $\widetilde{x}_0$. To see this, let $\widetilde{x}_2,\widetilde{x}_2'\in \operatorname{\text{Coin}}(q,f_i)$ be such that $q(\widetilde{x}_2)=q(\widetilde{x}_2')$,  and for which $\widetilde{x}_2$ (resp.\ $\widetilde{x}_2'$) belongs to the same Nielsen coincidence class of $(q,f_i)$ as $\widetilde{x}_0$ (resp.\ as $\widetilde{x}_1$). Let $\lambda\colon\thinspace I \to \widehat{X}$ be a path such that $\lambda(0)=\widetilde{x}_2'$,  $\lambda(1)=\widetilde{x}_1$ and $q\circ \lambda$ is homotopic to $f_i\circ \lambda$ relative to their endpoints, and let $\lambda'\colon\thinspace I \to \widehat{X}$ be the unique lift of $q\circ \lambda$ for which $\lambda'(0)=\widetilde{x}_2$, and let $\widetilde{x}_1'=\lambda'(1)$. Then $q(\widetilde{x}_1')=x_1$, and arguing as in the proof of \relem{fund1}, it follows that $q\circ \lambda'$ is homotopic to $f_i\circ \lambda'$ relative to their endpoints. Thus $\widetilde{x}_2$ and $\widetilde{x}_1'$ belong to the same Nielsen coincidence class of $(q,f_{i})$, which proves the claim since $\widetilde{x}_2$ and $\widetilde{x}_0$ belong to the same Nielsen coincidence class of $(q,f_i)$. Then by \relem{fund1}(\ref{it:fund1b}) and~(\ref{it:fund1c}), we have $\lvert K_i(\widetilde{x}_0)\rvert=\lvert K_i(\widetilde{x}_2)\rvert=\lvert K_i(\widetilde{x}_2')\rvert=\lvert K_i(\widetilde{x}_1)\rvert$ as required. 
\end{proof}


The following lemma is a key ingredient in the process of giving a computable formula for $N(\phi)$, which is the number of essential Nielsen classes.
We fix once and for all an orientation of the manifold $X$, and we choose the unique orientation of $\widehat{X}$ for which $q$ preserves orientation. 
Further,  the orientation on $\widehat{X}$ (resp.\ $X$) induces an orientation on any open subset of $\widehat{X}$ (resp.\ of $X$), and hence a local orientation system on $\widehat{X}$ (resp.\ on $X$). The map $q$ carries the local orientation of $\widehat X$ to that of $X$. The fixed point index of maps defined on open sets of $X$ as well as the coincidence index of maps from open sets of $\widehat{X}$ to $X$, are computed with respect to these orientations.

\begin{lem}\label{lem:index} Let $X$ be an orientable manifold,  
$x_0\in \operatorname{\text{Fix}}(\phi)$ be an isolated fixed point of $\phi$, and let $\widetilde{x}_{0}\in \widehat{X}$ and $i\in \brak{1,\ldots,n}$ be such that $q(\widetilde{x}_{0})=x_0$ and $\widetilde{x}_{0}\in \operatorname{\text{Coin}}(q,f_i)$. Then the
fixed point index of $\phi$ at $x_0$ is equal to the
coincidence index of the pair $(q, f_i)$ at $\widetilde{x}_{0}$.
\end{lem}

 \begin{proof}  Since
 $X$ is a manifold and $\widehat{X}$ is a finite covering of $X$, there exists an open, contractible neighbourhood $U$ of $x_0$ such that the restriction $q \left\lvert_{\widetilde{U}}\right.$ of $q$ to the component $\widetilde{U}$ of $q^{-1}(U)$ that contains $\widetilde{x}_0$ is a homeomorphism. The restriction of the $n$-valued map $\phi$ to $U$ is split, and a splitting  is given by $\brak{\overline{f}_1,\ldots ,\overline{f}_n}$, where  $\overline{f}_j=f_j\circ (q \left\lvert_{\widetilde{U}}\right.)^{-1}\colon\thinspace U \to X$ for $j=1,\ldots,n$.
 So $x_0$ is a fixed point of the map $\overline{f}_j$ for some $j\in \brak{1,\ldots,n}$,
and by definition, $\operatorname{\text{Ind}}(\phi, x_0)$ is equal to
$\operatorname{\text{Ind}}(\overline{f}_j, x_0)$. Since $q\left\lvert_{\widetilde{U}}\right.$ is an orientation-preserving
homeomorphism, $\operatorname{\text{Ind}}(\overline{f}_j, x_0)$ is equal to the coincidence index of the pair $(q, f_j)$ at the coincidence $\widetilde{x}_0$, where we use the local orientation in a neighbourhood of $\widetilde{x}_0$ determined by the local homeomorphism $q$, and this proves the lemma. 
 \end{proof}
 
One consequence of \relem{index} is that if a Nielsen coincidence class of $(q, f_i)$ is sent to a Nielsen fixed point class of $\phi$ then one of these Nielsen classes is essential if and only if the other is.  
 
\begin{cor}\label{cor:indexessent} 
Under the hypothesis of \relem{index}, if $C$ is a Nielsen coincidence class of the pair $(q, f_i)$ and $\widetilde{x}_0 \in C$, then the coincidence index of $C$ is equal to $\lvert K_i(\widetilde{x}_0)\rvert$ times the index of the fixed point class $q(C)$ of $\phi$. 
 \end{cor}
 
\begin{proof}
Let $C$ be a Nielsen coincidence class of the pair $(q, f_i)$ for some $i\in \brak{1,\dots,n}$. By \relem{index}, the points of $C$ that lie in $q^{-1}(q(\widetilde {x}_0))$ all have the same coincidence index for the pair $(q, f_i)$, this index being equal to $\operatorname{\text{Ind}}(\phi, x_0)$.
By \relem{auxiV}, the cardinality of the set $C\cap q^{-1}(q(\widetilde {x}_0))$   
is equal to $\lvert K_i(\widetilde{x}_0)\rvert$. 
If $\widetilde{x}_1$ is another point of $C$ (so is Nielsen equivalent to $\widetilde{x}_0$), the same conclusions hold, and similarly 
the cardinality of the set $C\cap q^{-1}(q(\widetilde {x}_1))$ is equal to $\lvert K_i(\widetilde{x}_0)\rvert$, and all of the points of this set have the same coincidence index for the pair $(q, f_i)$. 
Since the index of $C$ 
is the sum of the indices over the elements of the class $C$, which we can assume to be finite, 
the result follows.  
\end{proof}
 
For results related with \relem{index} and \reco{indexessent}  and more  results of similar nature, see
~\cite{Je,Moh}.


If $s\in \brak{1,\ldots,r}$ and $i\in I_{0}$, let $m_{i,s}= \lvert L_i\rvert/\lvert K_i(\widetilde{x}_s)\rvert$, where $\widetilde{x}_s$ is an element of one of the Nielsen coincidence classes of $(q,f_i)$ that belongs to the equivalence class $C_s$. This quantity is the number of Nielsen coincidence classes of $(q,f_i)$ in $C_s$ that are sent to the same Nielsen fixed point class of $\phi$ under the map $q$. Observe that if $\lvert L_i\rvert=1$ then $m_{i,s}=1$ for all $s\in \brak{1,\ldots,r}$.

We now come to the proof of the main result of this section.




%

\begin{proof}[Proof of \reth{form}]  Let $i_j \in \mathbb{O}_{j}$,
where $j$ runs over the set $I_0$ defined just before \relem{auxil}. From \relem{auxil0}, the image by $q$ of $\operatorname{\text{Coin}}(q,f_{i_j})$ is independent of the choice of representative $i_j$ in $\mathbb{O}_{j}$. Since the action of $L_{i,i}'$ on $\brak{1,\ldots,n}$ is free, $\lvert L_{i,i}'\rvert= \lvert L_i\rvert=1$, which implies that the map induced by $q$ between the Nielsen coincidence classes of $(q,f_i)$ and the Nielsen fixed point classes of $\phi$ is injective for all $i\in \brak{1,\ldots,n}$. Further, if $j,j'\in I_0$ are distinct then $q(\operatorname{\text{Coin}}(q,f_{i_j}))\cap q(\operatorname{\text{Coin}}(q,f_{i_{j'}}))=\vide$ by \relem{auxil0}, and we conclude that the map between the union of the Nielsen coincidence classes of the pairs $(q,f_{i_j})$, where $i_j\in I_0$, and the Nielsen fixed point classes of $\phi$ is injective too. This map is also surjective by \repr{nielsen0}(\ref{it:coinfix}) and \relem{nielsenclasses}. By \reco{indexessent}, a  coincidence class of $(q,f_{i_j})$ 
is essential if and only if its image under $q$ 
is essential. From this, it follows that $N(\phi)$ is equal to $\sum_{j\in I_0} N(q,f_{i_j})$.
\end{proof}

If the space $X$ is a non-orientable manifold, the situation is more complex, and it is not clear for the moment how to obtain a formula similar to that of \reth{form} in this case.
In the case $n=2$, the hypothesis of that theorem on the action of $L_i$ is always satisfied.

\begin{cor}\label{cor:nielsenphi}
Let $X$ be an orientable compact manifold without boundary and $\phi\colon\thinspace X \multimap X$  a non-split $2$-valued map. Then $N(\phi)=N(q, f_1)=N(q, f_2)$.
\end{cor} 

\begin{proof} 
Since $\phi$ is non-split, $L\cong L'\cong \Z_{2}$, there is a single orbit $\mathbb{O}_{1}=\brak{1,2}$, and $\lvert L_{1}\rvert=\lvert L_{2}\rvert=1$. Then by \reth{form}$, N(\phi)=N(q, f_j)$ for all $j\in \brak{1,2}$.
\end{proof}


\end{document}